\documentclass{article}
\usepackage{amsmath}
 \usepackage{graphics}
 \usepackage{epsfig}

\usepackage{amsthm,accents}
\newcounter{enumerConta}
\usepackage{amsfonts}
\usepackage{amsmath}
\usepackage{amssymb}
\usepackage{hyperref}
\usepackage{graphicx}
\usepackage{subfigure}
\usepackage{enumitem}
\newtheorem{theorem}{Theorem}[section]
\newtheorem{corollary}{Corollary}
\newtheorem{main}{Main Theorem}
\newtheorem{lemma}[theorem]{Lemma}

\theoremstyle{definition}
\newtheorem{definition}[theorem]{Definition}

\title{Equivalence between spectral properties of graphs with and without loops}
\author{Eleonora Andreotti$^*$ and Daniel Remondini and Armando Bazzani}
\begin{document}
\maketitle
\centerline{\scshape Eleonora Andreotti$^*$}
\medskip
{\footnotesize
% please put the address of the first author
 \centerline{Department of Physics, University of Torino, 10125 Torino, Italy}
  }
\medskip

\centerline{\scshape Daniel Remondini and Armando Bazzani}
\medskip
{\footnotesize
 % please put the address of the second  and third author
 \centerline{ Department of Physics and Astronomy (DIFA), University of Bologna, 40127 Bologna, Italy}
      \centerline{INFN Section of Bologna, Italy}
}

\bigskip

% The name of the associate editor will be entered by an editorial staff
% "Communicated by the associate editor name" is not needed for special issue.
 %\centerline{(Communicated by the associate editor name)}

%%%%%%%%%%%%%%%%%%%%%%

%\begin{frontmatter}

\begin{abstract}
In this paper we introduce a spectra preserving relation between graphs with loops and graphs without loops. This relation is achieved in two steps. First, by generalizing spectra results got on $(m, k)$-stars to a wider class of graphs, the $(m, k, s)$-stars with or without loops. Second, by defining a covering space of graphs with loops that allows to remove the presence of loops by increasing the graph dimension. The equivalence of the two class of graphs allows to study graph with loops as simple graph without loosing information.
\end{abstract}
%\begin{keyword}
%Graph loop \sep Multigraph \sep Graph reduction \sep Laplacian spectra \sep Laplacian matrices \sep Covering graph
%\MSC 15A18 \MSC 05C50 \MSC 05C75
%\end{keyword}
%\end{frontmatter}

\section{Introduction}

In graph theory there are many areas of social networks and biological networks in which multigraphs (i.e. graphs comprising loops and multiedges) arise more naturally than simple graphs \cite{wasserman1994social, BMSP:BMSP12, ROBINS2013261}. 
Moreover, these multigraph structures also emerge {when several types of graph homomorphisms are applied, such as} aggregation, scaling and blocking procedures, \cite{Shafie15, scott00, Ray:2014:GTA:2788177}.\\
Many extensively used approaches in network analysis only consider {simple graphs}: only single edges between two different vertices are allowed, and  self relations ({\it loops} \cite{biggs:1993, Godsil, Chung97}) are excluded. 
{Moreover, with respect to the graph Laplacian operator, we can consider the case of "generalized graph Laplacians" \cite{LaplEigGraphs07} in which the diagonal terms (corresponding to topological or weighted loops) can be considered as the "potential" of an Hamiltonian operator, useful for example when studying protein structure by network-based approaches starting from their Contact Maps \cite{ProtChemRev13, ProtMenic16}.}
% A  {\it loop} is the most common mathematical object for representing self relations where a vertex is both the sender and receiver of an edge,. \\
In practice, simple graphs are often derived from multigraphs by collapsing multiple edges into a single one and removing the loops \cite{Ray:2014:GTA:2788177, Bondy:1976:GTA:1097029}.
{This procedure is applied since many algorithms and theorems work only for simple graphs}, but these approaches may discard inherent information in the original network. 

%On the other hand, methods properly working for simple graphs does not guarantee to preserve the complete  information of the original multigraph; for example, methods based on the Laplacian matrix properties are useless to provide more information since the Laplacian matrix is invariant with respect presence of loops.\\
% As matter of principle, all the information of a multigraph are retained within the multiple and self relations of vertices. 
% However, this multistructure has not been treated in literature as extensively as well as the simple graphs.\\
In this manuscript, we propose a suitable method to treat graphs with loops as simple graphs, keeping the same eigenvalue spectrum and as much as possible the eigenvectors of their adjacency and transition matrices. 
One of the results of the paper is the possibility to associate a Laplacian matrix to a graph with loops that allows to study its topological properties.\\
Because eigenvalues and eigenvectors describe completely the matrix, by preserving the adjacency (or transition) matrix spectra (eigenvalues and eigenvectors) we maintain the informations and properties of the graphs as much as possible.
In this way, graphs with loops can be studied with tools extensively used for simple graphs.\\
To define the correspondence between graphs and simple graphs, we introduce an extension of the structure and the results discussed in \cite{Andreotti18}; then we build a correspondence between two classes of subgraphs, namely the $(m,k,s)-${\it star with loops} and the $(m,k,s)-${\it star} without loops.\\ %we give a strong tool in order to study graph with loops as a simple graph avoiding the discarding of information.
The paper is organized as follows: after some preliminary remarks (section \ref{sec:2}), in section \ref{sec:3} we generalize the class of $(m,k)$-star in graphs to a wider class of graph, the $(m,k,s)$-star with or without loop, in order to extend the results obtained in \cite{Andreotti18}. 
% By giving conditions on the graph structure which implies the presence of multiple eigenvalues. 
Then, we show a connection between eigenvectors and eigenvalues of graphs with  and  without loops. 
In particular, each vertex with loops can be described as an $(1,k,-)$-star with loop: we will give a useful tool to replace the looped vertices with an $(2,k,s)$-star without loop by maintaining the same spectrum.\\
Thanks to these results it is possible to describe a graph with loop by a graph without loop and to define the Laplacian matrix of the correspondence graph.
Finally, in section \ref{sec:4} we draw some conclusions and discuss an outlook on future developments.

\section{Preliminary definitions}\label{sec:2}

We consider an undirected weighted connected graph $\mathcal G:=(\mathcal V, \mathcal E, w)$, where  
the edges $\mathcal E$ connect $n$ vertexes $\mathcal V$ and $w$ is the edge weight function: $w:\mathcal E\rightarrow \mathbb R^{+}.$
Let $A$ be the weighted adjacency matrix, which is symmetric since the graph is undirected ($A\in Sym_n(\mathbb R^+)$),
$$A_{ij} = \begin{cases} w(i,j), & \mbox{if $i$ is connected to $j$ } (i\sim j) \\
0 & \mbox{otherwise }  \end{cases}$$
where $i,j\in\mathcal V$.\\
Since the graph is not necessarily simple, any diagonal element of $A$ could be nonzero.\\
If the graph $\mathcal G$ is simple,  we introduce the strength diagonal matrix $D$:
$$
D_{ij} = \begin{cases} \sum_{k=1}^n w(i,k), & \mbox{if } i= j\\
0 & \mbox{otherwise }  \end{cases}
$$
and we  define the Laplacian matrix $L\in Sym_n (\mathbb R)$ and normalized Laplacian matrix $\mathcal L\in Sym_n (\mathbb R)$ 
as
$$
L:=D-A, \quad \mathcal L:=D^{-1/2}(D-A)D^{-1/2}.
$$
Whenever we refer to the $k$-th eigenvalue of a Laplacian matrix, we will refer to the $k$-th nonzero eigenvalue according to an increasing order.

Furthermore, we observe that by defining the transition matrix $T$ as $T:=D^{-1}A$ ($T$ defines the transition probabilities of a random walk on the graph) ,
the eigenvalue spectrum  $\sigma(T)$ is related to $\sigma(\mathcal L)$. Indeed $T$ is similar to $\tilde A:=D^{-1/2}AD^{-1/2}$ via the invertible matrix $D^{1/2}$ and
and it is easy to prove that the following statements are equivalent
\begin{enumerate}[label=\textbf{S.\arabic*},ref=S.\arabic*]
\item $v$ is an eigenvector of $\tilde A$ with eigenvalue $\lambda$ 
\item $v^TD^{1/2}$ is a left eigenvector of $T$ with the eigenvalue $\lambda$
\item \label{tildeAL3}$D^{-1/2}v$ is a right eigenvector of $T$ with eigenvalue $\lambda$
\setcounter{enumerConta}{\value{enumi}}
\end{enumerate}
Then we consider the relation between $\sigma(\tilde A)$ and the spectrum $\sigma(\mathcal L)$ using the equivalence of
the following statements 
\begin{enumerate}[label=\textbf{S.\arabic*},ref=S.\arabic*]%[label=\textbf{S.\arabic*}]
\item \label{tildeAL1} $v$ is an eigenvector of $\tilde A$ with eigenvalue $\lambda$ 
\setcounter{enumi}{\value{enumerConta}}
\item \label{tildeAL4} $v$ is an eigenvector of $\mathcal L$ with the eigenvalue $1-\lambda.$
\end{enumerate}
This relation will be very useful later in order to link the Laplacian of graphs with and without loops.\\
%\begin{proof}
%Let $v$ be an eigenvector of $\tilde A$ with eigenvalue $\lambda$, so
%\begin{eqnarray}
%\tilde A v=\lambda v&\Leftrightarrow& D^{-1/2}AD^{-1/2} v =\lambda v\nonumber\\
%&\Leftrightarrow& \Big(I_n -D^{-1/2}AD^{-1/2}\Big) v=(1-\lambda) v \nonumber\\
%&\Leftrightarrow& \mathcal L v=(1-\lambda) v \nonumber
%\end{eqnarray}
%where $I_n$ is the identity matrix of order $n$.
%\end{proof}

For the classical results on Laplacian matrices and their application to network theory, one may refer to \cite{Chung97, Cohen_Havlin:2010, Newman:2010:NI:1809753, Anderson85, MERRIS1994143}.

\section{Definition of $(m,k,s)$-star with and without loop}\label{sec:3}

In the present section,
we define a wider class of weighted $(m,k)$-stars (we refer to it as the {\it weighted $(m,k,s)$-stars}), to generalize the results obtained on multiple eigenvalues of Laplacian matrices, transition and adjacency matrices\cite{Andreotti18}. Then we consider the problem of introducing a correspondence between the class of {\it weighted $(m,k,s)$-stars} and the class of
{\it weighted $(m,k,s)$-stars with loops}. In this way it is possible to remove 
loops from a {\it weighted $(m,k,s)$-stars with loops in graph} by replacing it with {\it weighted $(m,k,s)$-stars in the graph} of increasing size (the increase is the number of loops at most)
without changing the eigenvalue spectrum of adjacency and transition matrices. The section is divided in two subsections where we consider the problem of multiple eigenvalues 
for the adjacency and transition matrices of $(m,k,s)$-stars without and with loops.\\

\subsection{Eigenvalues multiplicity problem for $(m,k,s)$-star}

We recall that a $(m,k)$-star is a graph $\mathcal G=(\mathcal V, \mathcal E,w)$ whose vertex set $\mathcal V$ can be written as the disjointed union of two subsets $\mathcal V_1$ and $\mathcal V_2$ of cardinalities $m$ and $k$ respectively, such that the vertexes in $\mathcal V_1$ have no connections among them, and each of these vertexes is connected
with all the vertexes in $\mathcal V_2$: i.e
$$\forall i\in \mathcal V_1,\forall j\in \mathcal V_2,\quad (i,j)\in \mathcal E$$
$$\forall i,j\in \mathcal V_1, \quad (i,j)\notin \mathcal E.$$

the notation $(m,k)$-star denotes a graph with partitions of cardinality $|\mathcal V_1|=m$ and $|\mathcal V_2|=k$ by $S_{m,k}.$
To extend this definition we weaken the conditions on the connections between the vertexes of $\mathcal V_1$:

\begin{definition}[$(m,k,s)$-star: $S_{m,k,s}$ ]
A $(m,k,s)$-star is a graph $\mathcal G=(\mathcal V, \mathcal E,w)$ whose vertex set $\mathcal V$ can be written as the disjointed union of two subsets $\mathcal V_1$ and $\mathcal V_2$ of cardinalities $m$ and $k$ respectively, $s$ is a number $s\in\{0,1\},$ such that 
$$\forall i\in \mathcal V_1,\forall j\in \mathcal V_2,\quad (i,j)\in \mathcal E$$
$$\mbox{if $s=0$ then }\forall i_1,i_2\in \mathcal V_1, i_1\neq i_2 , \quad  (i_1,i_2)\notin \mathcal E,$$
$$\mbox{ if $s=1$ then $\forall i_1,i_2\in \mathcal V_1, i_1\neq i_2 , \quad (i_1,i_2)\in \mathcal E$}$$. 

By $S_{m,k,s}$ we denote a $(m,k,s)$-star graph of subsets $\mathcal V_1$ and $\mathcal V_2$ of cardinalities $|\mathcal V_1|=m$ and $|\mathcal V_2|=k$.
\end{definition}
In Fig.\ref{smk_star} are shown examples of $(m,k,1)$-star graph and $(m,k,0)$-star graph.

We define a \textit{$(m,k,s)$-star of a graph} $\mathcal G=(\mathcal V, \mathcal E,w)$ as the $(m,k,s)$-star of partitions $\mathcal V_1$, $\mathcal V_2\subset \mathcal V$ such that only the vertexes in $\mathcal V_2$ can be connected with the rest of the graph $\mathcal V\setminus(\mathcal V_1\cup\mathcal V_2)$: i.e.\\
$$\forall i\in \mathcal V_1,\forall j\in \mathcal V_2,\quad (i,j)\in \mathcal E$$
$$\mbox{ if $s=0$ then }\forall i \in \mathcal V_1, \forall j\in \mathcal V\setminus\mathcal V_2,  \quad (i,j)\notin \mathcal E, $$

$$ \mbox{ if $s=1$ then } \forall i \in \mathcal V_1, \forall j\in \mathcal V\setminus(\mathcal V_1 \cup \mathcal V_2),  \quad (i,j)\notin \mathcal E$$
$$\quad\quad \mbox{ and } \forall i,j\in \mathcal V_1, i\neq j, \quad (i,j)\in \mathcal E.$$

\begin{figure}[!!h]\label{smk_star}
%\begin{subfigure}{}
%\includegraphics[width=6cm]{S16.png}\end{subfigure}
\begin{subfigure}{}
\includegraphics[width=5cm]{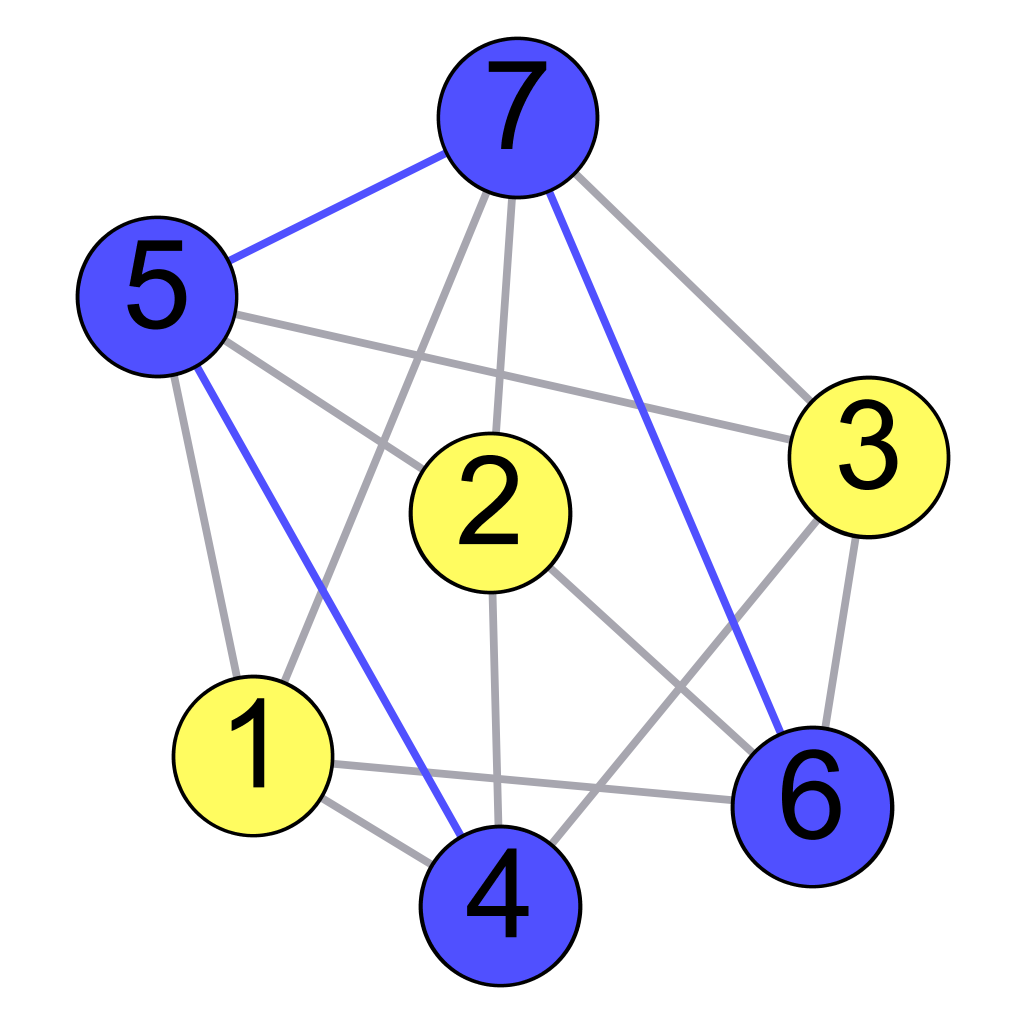}\end{subfigure}
%\begin{subfigure}{}
%\includegraphics[width=6cm]{S63.png}\end{subfigure}
\begin{subfigure}{}
\includegraphics[width=5cm]{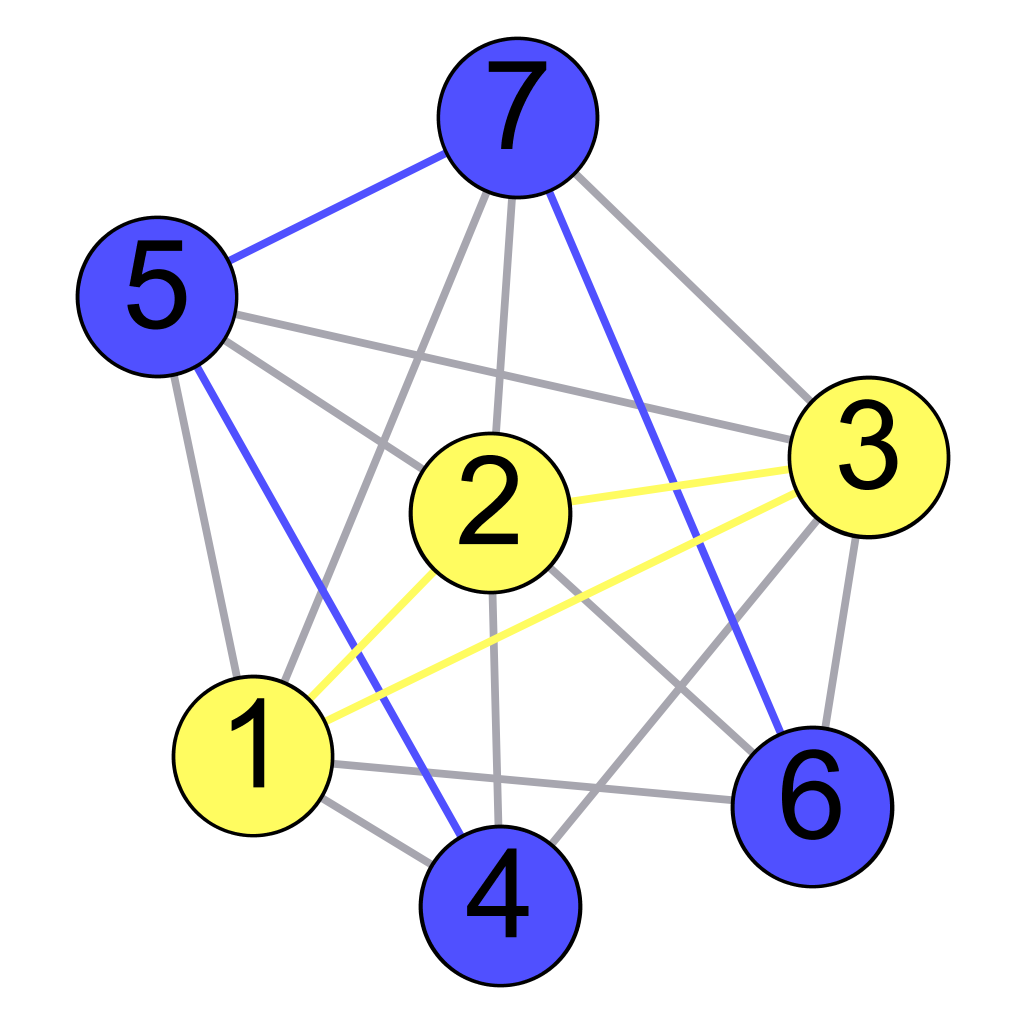}\end{subfigure}
\caption{Left: a $S_{3,4,0}$ graph; right: a $S_{3,4,1}$ graph. The subsets of vertices $\mathcal V_1$ and $\mathcal V_2$ are respectively colored in yellow and blue. 
Grey edges are between vertices belonging to different sets, yellow edges are between vertices in $\mathcal V_1$, and blue edges are between vertices in $\mathcal V_2$}
\end{figure}

By defining the concepts of degree, weight and central weight of a $(m,k,s)$-star we simplify the statement of the theorems on eigenvalues multiplicity.
\begin{definition}[Degree of a $(m,k,s)$-star: $deg(S_{m,k,s})$]
The \textit{degree} of a $(m,k,s)$- star is $deg(S_{m,k,s}):=m-1$ and the degree of a set $\mathcal S$ of $(m,k,s)$-stars, as $m$ and $k$ vary in $\mathbb N$, such that $|\mathcal S|=l,$ is defined as the sum over each $(m,k,s)$-star degree, i.e.
$$deg(\mathcal S):=\sum_{i=1}^l deg(S_{m_i,k_i,s_i}).$$
\end{definition}

\begin{definition}[Weight of a $(m,k,s)$-star: $w(S_{m,k,s})$]
The \textit{weight} of a $(m,k,s)$-star of vertices set $\mathcal V_1\cup\mathcal V_2$ is defined as follows:\\
let $\{i_1,...,i_m\}=\mathcal V_1$, and $w(i_1,j)=...=w(i_m,j), \forall j\in\mathcal V_2$ where all the vertices in $\mathcal V_1$ are connected to each other by links with the same weight, $w(i_p,i_1)=...=w(i_p,i_{p-1})=w(i_p,i_{p+1})=...=w(i_p,i_{m}), \forall i_p\in\mathcal V_1$, then we denote the weight of a $(m,k,s)$-star by $w(S_{m,k,s})$: $$w(S_{m,k,s}):=\sum_{j\in \mathcal V}w(i,j)
\mbox{ for any }i\in\mathcal V_1.$$
\end{definition}

\begin{definition}[Central weight of a $(m,k,s)$-star: $w_c(S_{m,k,s})$]\label{defn:cent}
The \textit{central weight} of a $(m,k,s)$-star of vertices set $\mathcal V_1\cup\mathcal V_2$ is defined as follows:\\
let $\{i_1,...,i_m\}=\mathcal V_1$, and $w(i_1,j)=...=w(i_m,j), \forall j\in\mathcal V_2$ where all the vertices in $\mathcal V_1$ are connected to each other by links with the same weight, $w(i_p,i_1)=...=w(i_p,i_{p-1})=w(i_p,i_{p+1})=...=w(i_p,i_{m}), \forall i_p\in\mathcal V_1$, then we denote the central weight of a $(m,k,s)$-star by $w_c(S_{m,k,s})$: $$w_c(S_{m,k,s}):=w(i,\tilde i)+w(i,i)\mbox{ for any }i,\tilde i\in\mathcal V_1, \ i\neq \tilde i.$$
\end{definition}
In the previous definition the weight of a loop, $w(i,i)$, is clearly set to zero and it is not considered in the present section.\\
Given a graph $\mathcal G=(\mathcal V,\mathcal E,w)$ associated with the Laplacian matrix $L$, and denoting $\sigma(L)$ the set of the
eigenvalues of $L$ and $m_L(\lambda)$ the algebraic multiplicity of the eigenvalue $\lambda$ in $L$, the following theorem holds, which extends the results in \cite{Andreotti18}\\
\begin{theorem}\label{th:one}
Let
\begin{itemize}
\item $S_{m,k,s}$, as $m$ and $k$ vary in $\mathbb N$  and $m+k\leq n,$ be each $(m,k,s)-star$ of $\mathcal G$;
\item  $r$ be the number of $S_{m,k,s}$ with different weight, $w_1,...,w_r$, i.e. $w_i\neq w_j$ for each $i\neq j,$ where $ i,j\in\{1,...,r\};$\\
\end{itemize}
then for any $i\in\{1,...,r\},$
$$\exists \lambda\in{\sigma(L)} \mbox{ such that } \lambda=w_i \mbox{ and } m_{L}(\lambda)\geq deg(\mathcal S_{w_i})$$
where $\mathcal S_{w_i}:=\{S_{m,k,s}\in \mathcal G | w(S_{m,k,s})+w_c(S_{m,k,s})=w_i\}$.
\end{theorem}

In order to prove our statement, we use the following Lemma on weighted adjacency matrix $A$:
\begin{lemma}\label{lemma:one}
let
\begin{itemize}
\item $S_{m,k,s}$, as $m$ and $k$ vary in $\mathbb N$  and $m+k\leq n,$ be each $(m,k,s)-star$ of $\mathcal G$;
\item  $r$ be the number of $S_{m,k,s}$ with different weight, $w_1,...,w_r$, i.e. $w_i\neq w_j$ for each $i\neq j,$ where $ i,j\in\{1,...,r\};$\\
\end{itemize}
then for any $i\in\{1,...,r\},$
$$\exists \lambda\in{\sigma(A)} \mbox{ such that } \lambda=-w_i \mbox{ and } m_{A}(\lambda)\geq deg(\mathcal S_{w_i})$$
where $\mathcal S_{w_i}:=\{S_{m,k,s}\in \mathcal G | w_c(S_{m,k,s})=w_i\}$.
\end{lemma}
where $\sigma(A)$ denotes the spectrum of $A$ and $m_A(\lambda)$ is the algebraic multiplicity of $\lambda$.
\begin{proof}
Without loss of generality we consider only connected graphs; 
indeed, if a graph is not connected the same result holds, since the $(m,k,s)$-star degree of the graph is the sum of the star degrees of the connected components and the characteristic polynomial of $L$ is the product of the characteristic polynomials of the connected components. \\
Let a $(m,k,s)$-star of the graph $\mathcal G$. 
Under a suitable permutation of the rows and columns of the weighted adjacency matrix $A$, we can label the vertexes in $\mathcal V_1$ with the indexes $1,...,m$, and the vertexes in $\mathcal V_2$ with the indexes $m+1,...,m+k$.\\
Let $v_1(A),...,v_m(A)$ be the rows corresponding to vertexes in $\mathcal V_1$, then the adjacency matrix has the following form
\[ A = \left( \begin{array}{cccc|cccccc}
 0 & w(1,2) &... & w(1,m) &  w(1,m+1)  &...& w(1,m+k)& 0&...&0\\
w(2,1) & 0 & \ddots & \vdots & \vdots  &...& \vdots & 0&...&0\\
 \vdots &\ddots & \ddots & w(m-1,m) &  \vdots & ...& \vdots & 0&...&0\\
 w(m,1)  & ...& w(m,m-1) & 0 &  w(m,m+1)  &...& w(m,m+k)& 0&...&0\\
\hline
w(1,m+1)& ...& ...  &  w(m,m+1)  &  &  & & & &\\
  \vdots & ... & ... & \vdots &  &  & & & &\\
 w(1,m+k)& ... & ... &  w(m,m+k) &  &  & & & & \\
 0& ... & ... & 0 &  &  & & & & \\
 \vdots& ... & ... & \vdots &  &  A_{22} & & & \\
 0& ... & ...&  0 &  &  & & & &  \\
\end{array} \right)\]
where the block $A_{22}$ is any $(n-m)\times(n-m)$ symmetric matrix with zero diagonal and nonnegative elements.\\ 
Because $w(1,2)=...=w(1,m)=w(2,3)=...=w(2,m)=...=w(m-1,m)=w_c(S_{m,k,s})$ the matrix $\hat A:=A+w_c(S_{m,k,s})I_n$ has $m$ rows (and $m$ columns) $v_1(\hat A),...,v_m(\hat A)$ linearly dependent such that $v_1(\hat A)=...=v_m(\hat A)$, then $v_1(\hat A),...,v_{m-1}(\hat A)\in ker(\hat A)$.\\
Hence
$$\exists \mu_1,...,\mu_{m-1}\in\sigma(\hat A)\quad \mbox{ such that }\quad \mu_1=...=\mu_{m-1}=0.$$
Let $\mu_i$ be one of these eigenvalues, then
$$
0=det((A+w_c(S_{m,k,s}) I_n)-\mu_i I_n)=det(A-(-w_c(S_{m,k,s}) +\mu_i )I_n)
$$
so that $\lambda:=-w_c(S_{m,k,s})\in\sigma (A)$ with multiplicity greater or equal to $deg(S_{m,k,s})$.\\
Let $p$ be the number of $S_{m,k,s}$ in the graph $\mathcal G$ that we indicate by $S_{m_1,k_1,s_1},...,S_{m_p,k_p,s_p}$.
Denoting $w^1_c,...,w^r_c$ the different central weights of such $(m,k,s)$-stars, and $r\leq p$, we prove that for any $i\in\{1,...,r\},$\\
$$
\exists \lambda\in{\sigma(A)} \mbox{ such that } \lambda=-w^i_c 
$$
$$
\mbox{ and the multiplicity of } \lambda\geq deg(\mathcal S_{w_c^i})=
\sum_{S_{m_j,k_j}\in\mathcal S_{w_c^i}} deg(S_{m_j,k_j,s_j}),
$$
where $\mathcal S_{w_c^i}:=\{S_{m,k,s}\in \mathcal G | w_c(S_{m,k,s})=w_c^i\}$.

Let $R_i$, with $i\in\{1,...,r\}$, be the number of $(m,k,s)$-stars in $\mathcal  S_{w_c^i}$, and $\sum_{i=1}^r R_r=p$,
we assume that the first $R_1$ indexes, namely $1,...,R_1$, refer to the $(m,k,s)$-stars in $\mathcal  S_{w_c^1}$,  where the indexes $R_1+1,...,R_1+R_2$ refer to the $(m,k,s)$-stars in
$\mathcal  S_{w_c^2}$, and so on.

We focus on the $R_i$ $(m,k,s)$-stars in $\mathcal  S_{w_c^i}$.
The rows in $\hat A:=A+w_c^iI_n$ corresponding to the vertexes $x_j$ in $\mathcal V_1(\mathcal S_{w_c^i})$ with $j\in\{\sum_{q=1}^{i-1} R_q+1,...,\sum_{q=1}^{i} R_q\}$,
are $m_j$ vectors $(v^{(j)}_{j_1}(\hat A),...,v^{(j)}_{j_{m_j}}(\hat A))$, linearly dependent and such that $v^{(j)}_{j_1}(\hat A)=...=v^{(j)}_{j_{m_j}}(\hat A)$,
whose indexes are
$$
j_1=\sum_{t=1}^{j-1} m_{t}+1,...,{j_{m_j}}=\sum_{t=1}^{j-1} m_{t}+m_j
$$
when $j>1$, or
$$
j_1=1,...,{j_{m_j}}=m_j
$$
when $j=1$.\\
Then we get
$$v^{(j)}_{j_1}(\hat A),...,v^{(j)}_{j_{{m_j}-1}}(\hat A)\in ker(\hat A),\quad \forall j\in\{\sum_{q=1}^{j-1} R_q+1,...,\sum_{q=1}^{j} R_q\}$$
and
$$\exists \mu_{j_1},...,\mu_{j_{{m_j}-1}}\in\sigma(\hat A)\quad \mbox{ such that }\quad \mu_{j_1}=...=\mu_{j_{{m_j}-1}}=0.$$

This is true for each $j\in\{\sum_{q=1}^{j-1} R_q+1,...,\sum_{q=1}^{j} R_q\}$, so that
$$\exists \mu_1,...,\mu_{deg(\mathcal S_{w_c^i})} \in\sigma(\hat A)\quad \mbox{ such that }\quad \mu_1=...=\mu_{deg(\mathcal S_{w^i_c})}=0.$$

%and the matrix $\hat A$ has at least $deg(\mathcal S_{w^i_c})+R_i$ diagonal entries with value $w_c^i$.\\

%In the matrix $(A+w_c^i I_n)$ there are $v^{(j)}_{j_q}(A+w_c^i I_n), \ q\in\{1,...,m_j\}$ vectors linearly dependent for each $j$, as a consequence
% $v^{(j)}_{j_1}(A+w_c^i I_n),...,v^{(j)}_{j_{m_j-1}}(A+w_c^i I_n)\in ker(A+w_c^i I_n)$ and
%$$\exists \mu_1,...,\mu_{deg(\mathcal S_{w_i})} \in\sigma(A+w_i I)\quad \mbox{ such that }\quad \mu_1=...=\mu_{deg(\mathcal S_{w_i})} =0.$$

Finally, let $\mu_t$ be one of these eigenvalues, then

$$0=det((A+w_c^i I_n)-\mu_t I_n)=det(A-(-w_c^i +\mu_t )I_n)$$

and $\lambda:=-w_c^i\in\sigma (A)$ with multiplicity greater or equal to $deg(\mathcal S_{w_c^i})$.\\
\end{proof}

The proof for the Laplacian version of the Lemma \ref{lemma:one} is similar to that for the adjacency matrix: using the same arguments as in the proof of \ref{lemma:one} we can say that the Theorem \ref{th:one} is true. \\
In Fig.\ref{smk_2} an example of a graph with an $(m,k,s)$-stars is shown. In this example the Laplacian matrix has an eigenvalue $\lambda=6$ with multiplicity 2.

Some corollaries on the signless and normalized Laplacian matrices can be obtained by using similar proofs.
Let $B$ and $\mathcal L$ be the signless and normalized Laplacian matrices of $\mathcal G=(\mathcal V,\mathcal E,w)$ respectively
and $\sigma(B)$, $\sigma(\mathcal L)$ the spectrum of $B$ and $\mathcal L$ with algebraic multiplicity
$m_B(\lambda)$, $m_{\mathcal L}(\lambda)$ for the eigenvalue $\lambda$ in $B$ and $\mathcal L$ respectively.\\
\begin{corollary}
If
\begin{itemize}
\item $S_{m,k,s}$, as $m$ and $k$ vary in $\mathbb N$  and $m+k\leq n,$ is each $(m,k,s)-star$ of $\mathcal G$;
\item $r$ is the number of $S_{m,k,s}$ with different weights, $w_1,...,w_r$,\\
\end{itemize}
then for any $i\in\{1,...,r\},$
$$\exists \lambda\in{\sigma(B)} \mbox{ such that } \lambda=w_i \mbox{ and } m_B(\lambda)\geq deg(\mathcal S_{w_i})$$
where $\mathcal S_{w_i}:=\{S_{m,k,s}\in \mathcal G | w(S_{m,k,s})-w_c(S_{m,k,s})=w_i\}$.
\end{corollary}

\begin{corollary}
If
\begin{itemize}
\item $S_{m,k,s}$, as $m$ and $k$ vary in $\mathbb N$  and $m+k\leq n,$ is each $(m,k,s)-star$ of $\mathcal G$;
\item $r$ is the number of $S_{m,k,s}$ with different weights, $w_1,...,w_r$,\\
\end{itemize}
then for any $i\in\{1,...,r\},$
$$\exists \lambda\in{\sigma(\mathcal L)} \mbox{ such that } \lambda=1+w_i \mbox{ and } m_{\mathcal L}(\lambda)\geq \sum_{i=1}^r deg(\mathcal S_{w_i})$$
where $\mathcal S_{w_i}:=\{S_{m,k,s}\in \mathcal G | \displaystyle\frac{w_c(S_{m,k,s})}{w(S_{m,k,s})}=w_i\}$.
\end{corollary}
From relations \ref{tildeAL1}, \ref{tildeAL4}, in the above situation we also have that
$$\exists \lambda\in{\sigma( T)} \mbox{ such that } \lambda=w_i \mbox{ and } m_{T}(\lambda)\geq \sum_{i=1}^r deg(\mathcal S_{w_i}).$$

\begin{figure}[!!h]\label{smk_2}
\centering
\includegraphics[width=8cm]{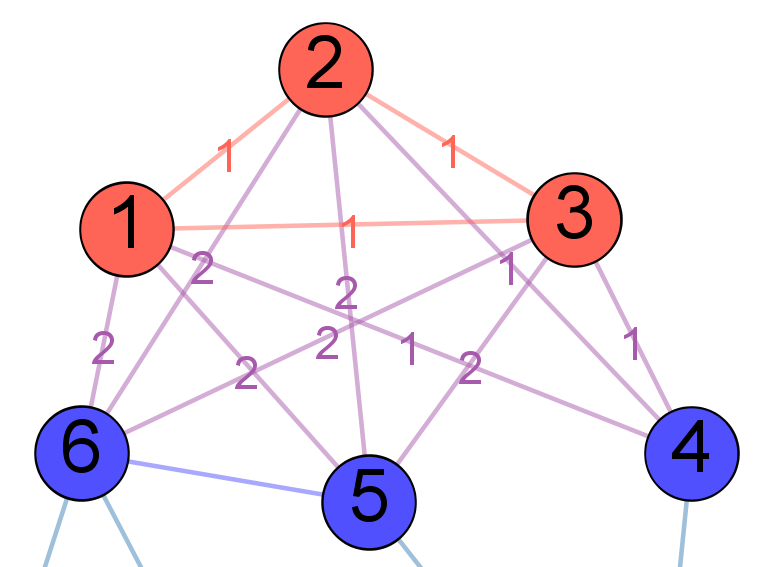}\caption{A $S_{3,3,1}$ in a graph, where the subsets $\mathcal V_1$ (red vertices) and $\mathcal V_2$ (blue vertices) are respectively with cardinality $m=3$ and $ k=3$. 
The weights of the edges between vertices belonging to $\mathcal V_1$ are colored in red, the weights of the edges between vertices belonging to two different sets are colored in purple. 
In the Laplacian matrix there is an eigenvalue $\lambda=6$ with multiplicity 2.}
\end{figure}

We observe that when $s=0$, and thus $w_c=0$, each of the above results can be reduced to the results obtained in \cite{Andreotti18}.
%%%%%%%%%%%%%%%%%%%%%%%%%%%%%%%%%%%%%

\subsection{Eigenvalues multiplicity problem for $(m,k,s)$-star with loops}

In this section we consider $(m,k,s)$-star with loops and we generalize the results discussed in the previous section.
Some definitions are useful:

\begin{definition}[$(m,k,s)$-star with loop: $\accentset{\circ}{S}_{m,k,s}$ ]

A $(m,k,s)$-star with loops is a $(m,k,s)$-star in which each vertex in the set $\mathcal V_1$ has a loop.
A $(m,k,s)$-star denotes a graph with partitions of cardinality $|\mathcal V_1|=m$ and $|\mathcal V_2|=k$ by $\accentset{\circ}{S}_{m,k,s}.$
\end{definition}

We define a \textit{$(m,k,s)$-star with loops of a graph} $\mathcal G=(\mathcal V, \mathcal E,w)$ as the $(m,k,s)$-star with loop of partitions $\mathcal V_1$, $\mathcal V_2\subset \mathcal V$ such that 
$$\forall i\in \mathcal V_1,\forall j\in \mathcal V_2\cup\{i\},\quad (i,j)\in \mathcal E$$
$$\mbox{ if $s=0$ then }\forall i \in \mathcal V_1, \forall j\in \mathcal V\setminus(\mathcal V_2\cup\{i\}),  \quad (i,j)\notin \mathcal E, $$

$$ \mbox{ if $s=1$ then } \forall i \in \mathcal V_1, \forall j\in \mathcal V\setminus(\mathcal V_1 \cup \mathcal V_2),  \quad (i,j)\notin \mathcal E$$
$$\quad\quad \mbox{ and } \forall i,j\in \mathcal V_1, \quad (i,j)\in \mathcal E$$

In other words, a \textit{$(m,k,s)$-star with loops of a graph} $\mathcal G=(\mathcal V, \mathcal E,w)$ is a $(m,k,s)$-star of a graph $\mathcal G$ in which each vertex in the set $\mathcal V_1$ has a loop.

By defining the degree, weight and central weight of a $(m,k,s)$-star with loop as in the previous section we simplify the stating of the theorems on eigenvalues multiplicity.
For the $(m,k,s)$-stars with loops the Lemma \ref{lemma:one} is modified as follows
\begin{lemma}\label{lemma:two}
Let
\begin{itemize}
\item $\accentset{\circ}{S}_{m,k,s}$, as $m$ and $k$ vary in $\mathbb N$  and $m+k\leq n,$ be each $(m,k,s)-star$ with loops  of $\mathcal G$;
\item  $r$ be the number of $\accentset{\circ}{S}_{m,k,s}$ with different weights, $w_1,...,w_r$, i.e. $w_i\neq w_j$ for each $i\neq j,$ where $ i,j\in\{1,...,r\};$\\
\end{itemize}
then for any $i\in\{1,...,r\},$
$$\exists \lambda\in{\sigma(A)} \mbox{ such that } \lambda=-w_i \mbox{ and } m_{A}(\lambda)\geq deg(\mathcal S_{w_i})$$
where $\mathcal S_{w_i}:=\{\accentset{\circ}{S}_{m,k,s}\in \mathcal G | w_c(\accentset{\circ}{S}_{m,k,s})=w_i\}$.
\end{lemma}

For graphs with loops we can't apply the results on spectra of Laplacian matrices, but we can prove a result for the transition matrix analogous to that for simple graphs.

\begin{corollary}
If
\begin{itemize}
\item $\accentset{\circ}{S}_{m,k,s}$, as $m$ and $k$ vary in $\mathbb N$  and $m+k\leq n,$ is each $(m,k,s)-star$ with loops of $\mathcal G$;
\item $r$ is the number of $\accentset{\circ}{S}_{m,k,s}$ with different weights, $w_1,...,w_r$,\\
\end{itemize}
then for any $i\in\{1,...,r\},$
$$\exists \lambda\in{\sigma(T)} \mbox{ such that } \lambda=-w_i \mbox{ and } m_{T}(\lambda)\geq \sum_{i=1}^r deg(\mathcal S_{w_i})$$
where $\mathcal S_{w_i}:=\{\accentset{\circ}{S}_{m,k,s}\in \mathcal G | \displaystyle\frac{w_c(\accentset{\circ}{S}_{m,k,s})}{w(\accentset{\circ}{S}_{m,k,s})}=w_i\}$.
\end{corollary}

\subsection{Correspondence between (m,k,s)-stars with loops and without loops}
In this section we define a correspondence between $(m,k,s)$-stars with loops and the $(m,k,s)$-stars without loops which preserves the eigenvalue spectrum of adjacency and transition matrices. 
In particular, in a graph, each vertex with a loop is equivalent to an $(1,k,-)$-star with loop and we will provide a procedure to replace the looped vertex with an $(2,k,s)$-star without loops 
which has the same spectra (see Fig.(\ref{looptono})).
\\

\begin{figure}[!!h]\label{looptono}
\begin{subfigure}{}
\includegraphics[width=4cm]{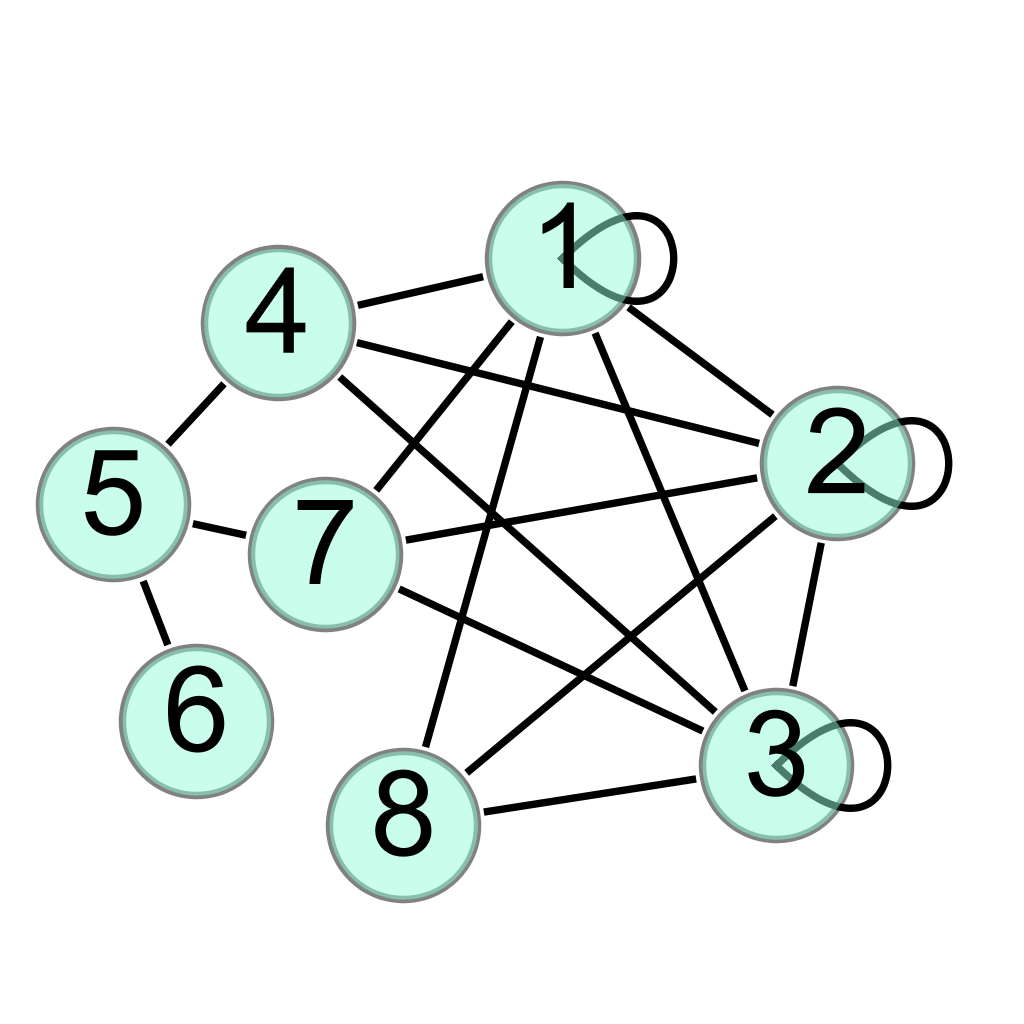}\end{subfigure}
\begin{subfigure}{}
\includegraphics[width=3.5cm]{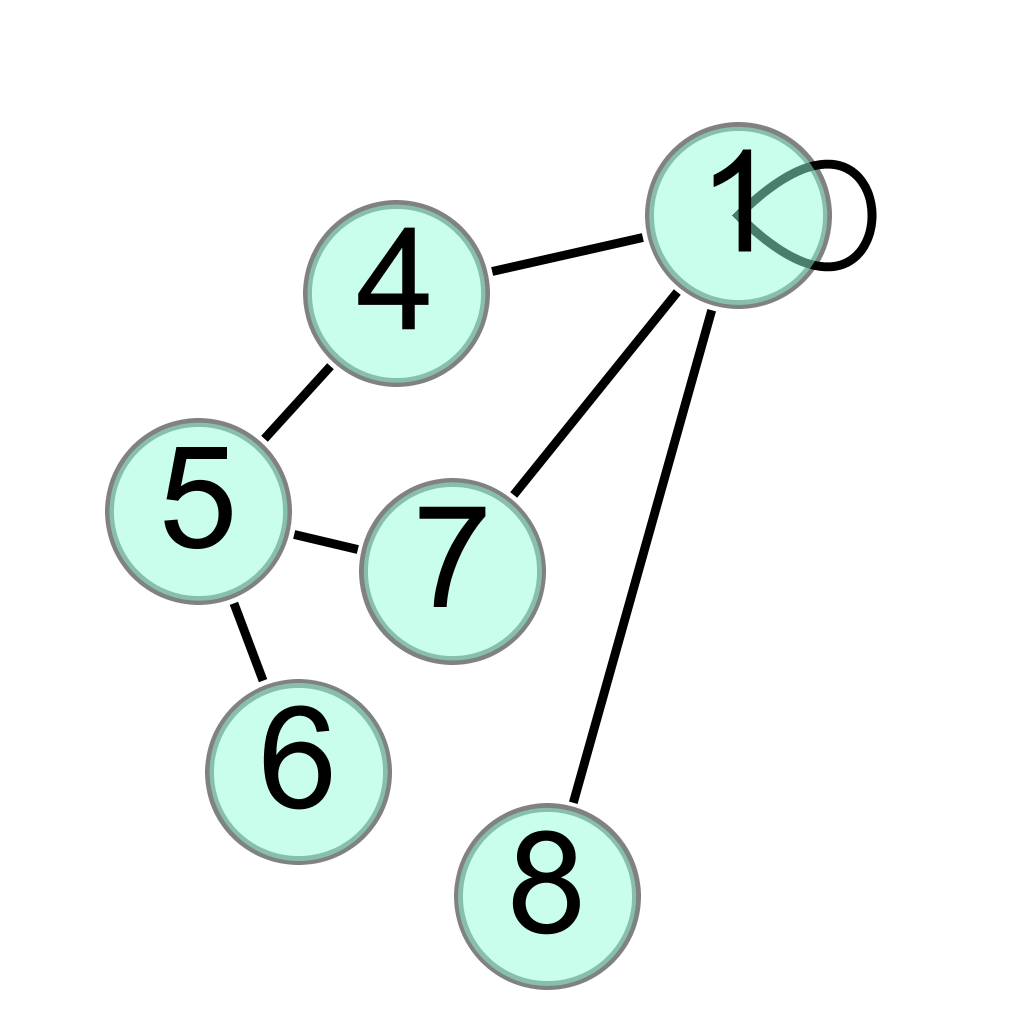}\end{subfigure}
\vspace{5mm}\begin{subfigure}{}
\includegraphics[width=3.5cm]{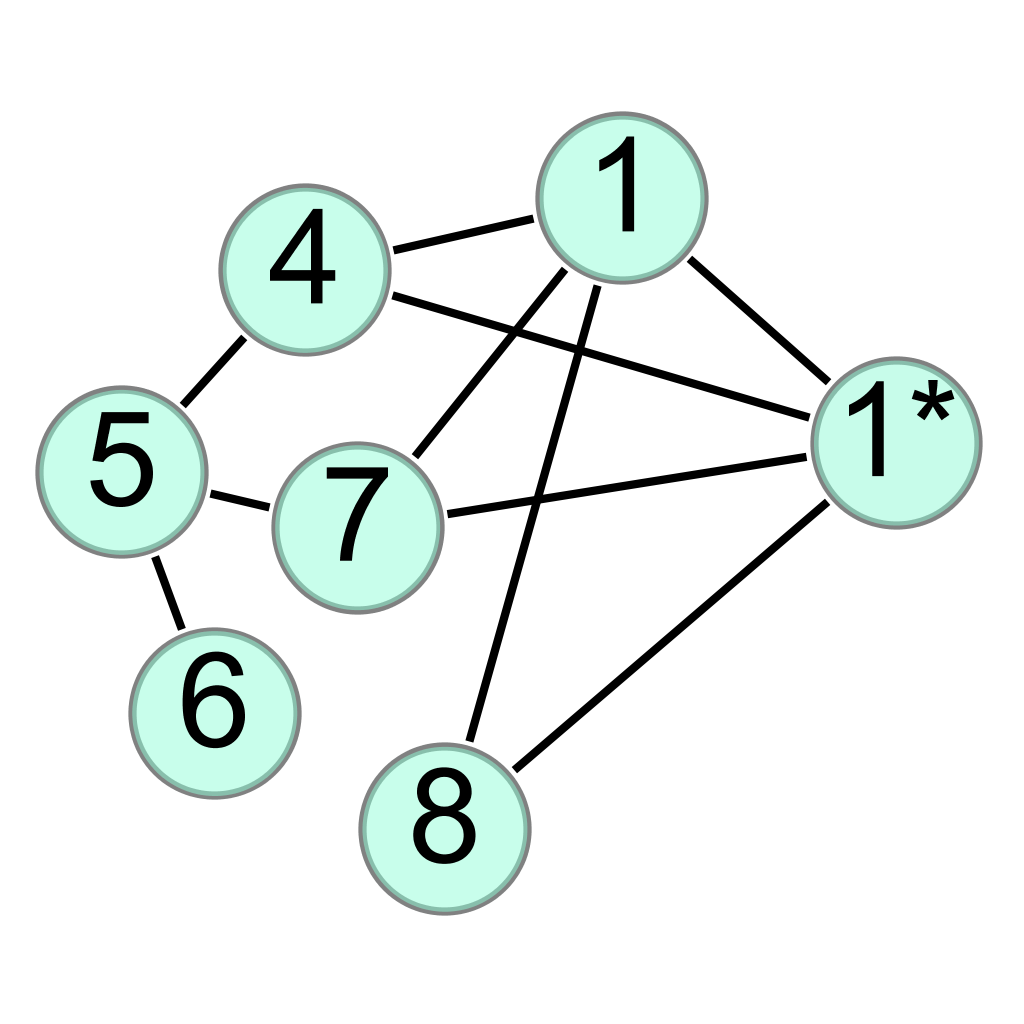}\end{subfigure}\vspace{-5mm}
\caption{Left: a $\mathring{S}_{3,3,1}$ in a graph; center: its $\mathring{S}^{\{-2\}}_{3,3,1}$ that is a $\mathring{S}_{1,3,-}$ in a graph; right: its $\mathring{S}^{\{+1\}}_{1,3,-}$ that is a ${S}_{2,3,1}$ in a graph.}
\end{figure}

The following definitions are useful:
\begin{definition}[$(m,k,s)$-star $q$-reduced: $S^{\{-q\}}_{m,k,s}$]
A $q$-reduced $(m,k,s)$-star is a $(m,k,s)$-star (with or without loops) of vertex sets $\{\mathcal V_1,\mathcal V_2\}$, such that the cardinality of $\mathcal V_1$ is decreased to $m-q$, with $m>q$.\\
Hence the order and degree of the $S^{\{-q\}}_{m,k,s}$ are $m+k-q$ and $m-q-1$ respectively.\\
Furthermore, let $w$ be the weights between vertexes in the original $(m,k,s)$-star and $\tilde i, \tilde j$ any vertex in $\mathcal V_1$, $\tilde i\neq \tilde j$, of the $(m,k,s)$-star, we define the weights of vertexes in $S^{\{-q\}}_{m,k,s}$ as 
\begin{equation}\label{eq:weightloop1}
w^{\{-q\}}(i,j) = \begin{cases}\frac{q}{m-q}w(\tilde i,\tilde j)+w(\tilde i,\tilde i),& \mbox{if } i,j\in\mathcal V_1,\  i=j\\
\frac{m}{m-q}w(\tilde i,j), & \mbox{if } i\in\mathcal V_1, j\in (\mathcal V_1\cup\mathcal V_2)\setminus\{i\} \\
 w(i,j), & \mbox{if } i,j \in\mathcal V_2 \\
0 & otherwise \end{cases},
\end{equation}
\end{definition}

\begin{definition}[$(1,k,-)$-star $q$-enlarged: $S^{\{+q\}}_{1,k,-}$]
A $q$-enlarged $(1,k,-)$-star is a $(1,k,-)$-star (with or without loops) of vertex sets $\{\mathcal V_1,\mathcal V_2\}$, such that the cardinality of $\mathcal V_1$ is increased to $q+1$ and the loop is removed.\\
Hence the order and degree of the $S^{\{+q\}}_{1,k,-}$ are $1+k+q$ and $q$ respectively.\\
Furthermore, let $w$ be the weights between vertexes in the original $(1,k,-)$-star and $\tilde i$ the vertex in $\mathcal V_1$ of the $(1,k,-)$-star, 
we define the weights of vertexes in $S^{\{+q\}}_{1,k,-}$ as
\begin{equation}\label{eq:weightloop2}
w^{\{+q\}}(i,j) = \begin{cases}\frac{1}{q}w(\tilde i,\tilde i), & \mbox{if } i\in\mathcal V_1, j\in \mathcal V_1\setminus\{i\} \\
\frac{1}{1+q}w(\tilde i,j), & \mbox{if } i\in\mathcal V_1, j\in \mathcal V_2 \\
 w(i,j) & \mbox{if } i,j \in\mathcal V_2 \\
0 & \mbox{if } i=j \in\mathcal V_1 \end{cases},
\end{equation}
\end{definition}

Finally we introduce the concept of the $q$-enlarged graph associate to a graph:
\begin{definition}[$q$-enlarged graph: $\mathcal G^{\{+q\}}$]
A $q$-enlarged graph $\mathcal G^{\{+q\}}$ is obtained from a graph $\mathcal G$ with some $(1,k,-)$-stars adding $q$ the vertexes in the set $\mathcal V_1$
of $\mathcal G$, removing the loops and defining the weights as in \eqref{eq:weightloop2}.
\end{definition}

% Similarly one can define the $q$-reduced graph. 
Using the previous definitions it is possible to associate to any graph containing $(m,k,s)-star$ with loops (or more simply whose vertexes has a loop),
an enlarged graph without loops  by means of an intermediate reduced graph  as illustrated in Fig.(\ref{looptono}). The following theorem holds

\begin{main}[Loop removal theorem - adjacency matrix]\label{th:loop_adj}
Let
\begin{itemize}
\item $\mathcal G$ be a graph, of n vertexes, with a $\accentset{\circ}{S}_{m,k,s}$,
\item $\mathcal H:=\mathcal G^{\{-(m-1)\}}$ be the $(m-1)$-reduced graph with a $\accentset{\circ}{S}_{m,k,s}^{\{-(m-1)\}}$ instead of $\accentset{\circ}{S}_{m,k,s}$,
\item $\mathcal I:=\mathcal H^{\{+q\}}$ be the $q$-enlarged graph with a ${S}_{1,k,-}^{\{+q\}}$ instead of $\accentset{\circ}{S}_{1,k,-}$,
\item $A$ be the adjacency matrix of $\mathcal G$,
\item $A^{\{-(m-1)\}}$ be the adjacency matrix of $\mathcal H$, defined as in \eqref{eq:weightloop1}
\item $A^{\{+q\}}$ be the adjacency matrix of $\mathcal I$, defined as in \eqref{eq:weightloop2}
\end{itemize}
then
\begin{enumerate}
\item $\sigma(A^{\{-(m-1)\}})\subset\sigma(A),$\\
\item $\sigma(A^{\{-(m-1)\}})\subset\sigma(A^{\{+q\}}),$\\
%where $w_c$ is the central weight defined in \ref{defn:cent};
\item There exists a matrix $ H\in\mathbb R^{n\times (n-(m-1))}$ such that $A^{\{-(m-1)\}}= H^TA H$ and $ H^T H=I$. 
Therefore, if $v$ is an eigenvector of $A^{\{-(m-1)\}}$ for an eigenvalue $\mu$, then $ Hv$ is an eigenvector of $A$ for the same eigenvalue $\mu$.
\item There exists a matrix $K\in\mathbb R^{n-(m-1)\times (n-(m-1)+q)}$ such that $A^{\{-(m-1)\}}=K^TA^{\{+q\}}K$ and $K^TK=I$. Therefore, if $v$ is an eigenvector of $A^{\{-(m-1)\}}$ for an
eigenvalue $\mu$, then $Kv$ is an eigenvector of $A^{\{+q\}}$ for the same eigenvalue $\mu$.
\end{enumerate}
\end{main}

{ Before proving Theorem \ref{th:loop_adj}, we recall the well known result for eigenvalues of symmetric matrices, \cite{Hwang2004}.

\begin{lemma}[Interlacing theorem]
Let $A\in Sym_{n_A}(\mathbb R)$ with eigenvalues $\mu_1(A)\geq...\geq \mu_{n_A}(A).$ For $n_B<n_A$, let $K\in\mathbb R^{n_A,n_B}$ be a
matrix with orthonormal columns, $K^TK=I$, and consider the  $B=K^TAK$ matrix, with eigenvalues $\mu_1(B)\geq...\geq \mu_{n_B}(B).$
If
\begin{itemize}
\item the eigenvalues of $B$ interlace those of $A$, that is,
$$\mu_i(A)\geq\mu_i(B)\geq\mu_{n_A-n_B+i}(A), \quad i=1,...,n_B,$$
\item if the interlacing is tight, that is, for some $0\leq k\leq n_B,$
$$\mu_i(A)=\mu_i(B), \ i=1,...,k\ \mbox{ and } \ \mu_i(B)=\mu_{n_A-n_B+i}(A), \ i=k+1,...,n_B$$
then $KB=AK.$
\end{itemize}
\end{lemma}

\begin{proof}
We will explicitly prove only the items 2. and 4., because using the same arguments the statements 1. and 3. follow and the matrix $ H$ exists.\\
First we prove the existence of the $K$ matrix:\\
let $\tilde n:=n-(m-1)$ and $\mathcal P=\{P_1,...,P_{\tilde n}\}$ be a partition of the vertex set $\{1,...,\tilde n+q\}$.
The \textit{ characteristic matrix $\tilde K$ } is defined as the matrix where the $j$-th column is the characteristic vector of $P_j$ ($j=1,...,\tilde n$).\\
Let $A^{\{+q\}}$ be partitioned according to $\mathcal P$
\[
A^{\{+q\}}=\left(
\begin{array}{ccc}
A^{\{+q\}}_{1,1} & \dots & A^{\{+q\}}_{1,\tilde n} \\
\vdots &  & \vdots \\
A^{\{+q\}}_{\tilde n,1} & \dots & A^{\{+q\}}_{\tilde n,\tilde n}
\end{array}
\right),\]

where $A^{\{+q\}}_{i,j}$ denotes the block with rows in $P_i$ and columns in $P_j$.
The matrix $A^{\{-(m-1)\}}=(a^{\{-(m-1)\}}_{ij})$ whose entries $a^{\{-(m-1)\}}_{ij}$ are the averages of the $A^{\{+q\}}_{i,j}$ rows, is called the \textit{quotient matrix} of $A^{\{+q\}}$ with respect to $\mathcal P$,
i.e. $a^{\{-(m-1)\}}_{ij}$ denotes the average number of neighbors in $P_j$ of the vertices in $P_i$.\\
The partition is equitable if for each $i,j$, any vertex in $P_i$ has exactly $a^{\{-(m-1)\}}_{ij}$ neighbors in $P_j$.
In such a case, the eigenvalues of the quotient matrix $A^{\{-(m-1)\}}$ belong to the spectrum of $A^{\{+q\}}$ ($\sigma(A^{\{-(m-1)\}})\subset\sigma(A^{\{+q\}})$) and the spectral radius
of $A^{\{-(m-1)\}}$ equals the spectral radius of $A^{\{+q\}}$: for more details cfr. \cite{brouwer12}, chapter 2.\\
Then we have the relations
$$
M^{1/2}A^{\{-(m-1)\}}M^{1/2}=\tilde K^TA^{\{+q\}}\tilde K, \quad \tilde K^T\tilde K=M.
$$
Considering an $(1,k,-)-$star (with loop) in a graph with adjacency matrix $A^{\{-(m-1)\}}$, we weight it by a diagonal mass matrix $M$ of order $\tilde n$ whose diagonal entries
are one except for the entry of the vertex in $\mathcal V_1$,
\begin{equation}\label{eq:diagM}
M_{ii} = \begin{cases} \frac{1}{1+q}, & \mbox{if } i\in\mathcal V_1 \\ 1 & \mbox{otherwise } \end{cases},
\end{equation}
and we get
$$ 
A^{\{-(m-1)\}}=K^TA^{\{+q\}}K, \quad K^TK=I,
$$
where $K:=\tilde KM^{-1/2}.$
In addition to the Theorem (\ref{th:one}), the eigenvalues of the matrix $A^{\{-(m-1)\}}$ belong also to the spectrum of the matrix $A^{\{+q\}}$,

$$
\sigma(A^{\{-(m-1)\}})\subset\sigma(A^{\{+q\}}).
$$

Finally, if $v$ is an eigenvector of $A^{\{-(m-1)\}}$ with eigenvalue $\mu$, then $Kv$ is an eigenvector of $A^{\{+q\}}$ with the same eigenvalue $\mu$.\\
Indeed, from the equation
$A^{\{-(m-1)\}}v=\mu v$
and taking into account that the partition is equitable, we have $KA^{\{-(m-1)\}}=A^{\{+q\}}K,$ and
$$A^{\{+q\}}(Kv)=(A^{\{+q\}}K)v=(KA^{\{-(m-1)\}})v=\mu (Kv).$$

\end{proof}}

A similar result holds for the transition matrix $T$, and more in general for each $D^{-1}A$ where $A$ is the adjacency matrix of the graph $\mathcal G$ with a $\accentset{\circ}{S}_{m,k,s}$ 
and $D$ any real diagonal matrix such that $d_{ii}=d_{jj}$ for any $i,j\in \mathcal V_1$. This is states by the following Theorem:\\
%%%%%%%%%%%%%%%%%%%%%%%%%%%%%%%%%%%%%%%%%%%%%%%%%%

\begin{theorem}[Loop removal theorem - transition matrix]\label{th:loop_tran}
Let
\begin{itemize}
\item $\mathcal G$ be a graph, of n vertices, with a $\accentset{\circ}{S}_{m,k,s}$,
\item $\mathcal H:=\mathcal G^{\{-(m-1)\}}$ be the $(m-1)$-reduced graph with a $\accentset{\circ}{S}_{m,k,s}^{\{-(m-1)\}}$ instead of $\accentset{\circ}{S}_{m,k,s}$,
\item $\mathcal I:=\mathcal H^{\{+q\}}$ be the $q$-enlarged graph with a ${S}_{1,k,-}^{\{+q\}}$ instead of $\accentset{\circ}{S}_{1,k,-}$,
\item $A$ and $D$ be, respectively, the adjacency matrix and the strength diagonal matrix of $\mathcal G$,
\item $A^{\{-(m-1)\}}$ and $D^{\{-(m-1)\}}$ be, respectively, the adjacency matrix and the strength diagonal matrix of $\mathcal H$, defined as in \eqref{eq:weightloop1}
\item $A^{\{+q\}}$ and $D^{\{+q\}}$ be, respectively, the adjacency matrix and the strength diagonal matrix of $\mathcal I$, defined as in \eqref{eq:weightloop2}
\end{itemize}
then
\begin{enumerate}
\item $\sigma(T^{\{-(m-1)\}})\subset\sigma(T),$ where $T^{\{-(m-1)\}}:=(D^{\{-(m-1)\}})^{-1}A^{\{-(m-1)\}}$ and $T:=D^{-1}A$\\
\item $\sigma(T^{\{-(m-1)\}})\subset\sigma(T^{\{+q\}}),$ where $T^{\{+q\}}:=(D^{\{+q\}})^{-1}A^{\{+q\}}$\\
%where $w_c$ is the central weight defined in \ref{defn:cent};
\item There exists a matrix $H\in\mathbb R^{n\times (n-(m-1))}$ such that $T^{\{-(m-1)\}}=H^TTH$ and $H^TH=I$. Therefore, if $v$ is an eigenvector of $T^{\{-(m-1)\}}$ for an
eigenvalue $\mu$, then $Hv$ is an eigenvector of $T$ for the same eigenvalue $\mu$.
\item There exists a matrix $K\in\mathbb R^{n-(m-1)\times (n-(m-1)+q)}$ such that $T^{\{-(m-1)\}}=K^TT^{\{+q\}}K$ and $K^TK=I$. Therefore, if $v$ is an eigenvector of $T^{\{-(m-1)\}}$ for an
eigenvalue $\mu$, then $Kv$ is an eigenvector of $T^{\{+q\}}$ for the same eigenvalue $\mu$.
\end{enumerate}
\end{theorem}
{The proof for the transition matrix version of the Loop removal theorem \ref{th:loop_adj} is similar to that for the adjacency matrix. More explicitly, using the same arguments as in the proof of
\ref{th:loop_adj} and the equivalences \ref{tildeAL1}--\ref{tildeAL3} in order to work with symmetric matrices, the items 1. and 2. are true and the matrices $H$ and $K$ exist. 
%%%%%%%%%%%%%%%%%%%%%%%%%%%%%%%%%%%%%%%

\begin{corollary}\label{cor:normlap}
Let $\mathcal G$ be a graph, of $n$ vertexes, with a $\accentset{\circ}{S}_{m,k,s}$,
if $v$ is a right eigenvector of $T$ with eigenvalue $\lambda\in
\sigma(T)\setminus\{\displaystyle-\frac{w_c(\accentset{\circ}{S}_{m,k,s})}{w(\accentset{\circ}{S}_{m,k,s})}\}$
then $D^{1/2}v$ is an eigenvector of $\mathcal L$ with eigenvalue $(1-\lambda)$.

\end{corollary}
The proof directly follows from the Theorem \ref{th:loop_tran}.

\section{Conclusions}\label{sec:4}

The Laplacian matrix associated to undirected graphs provides powerful tools to study the geometrical and dynamical properties of the graph \cite{biggs:1993, Chung97}. 
In particular, its spectral properties allow to study random walk processes on graphs (e.g. the existence of bifurcation phenomena in the solutions) {and to characterize normal modes in a "springs and masses" interpretation of the graph}. 
The possibility of associating a Laplacian matrix to multigraphs can be a powerful tool in the application of graph theory to network theory in complex system physics, {for example in the case of generalized graph Laplacians, in which we can consider both a "kinetic" and a "potential" energy term}.
In a previous work \cite{Andreotti18} we have associated the presence of $(m,k)$-stars in a graph to eigenvalue multiplicity in the Laplacian matrix spectrum. 
In this work, we have extended the previous results for $(m,k)$-stars to  $(m,k,s)$-stars with loops. 
Our approach allows to introduce relations between the spectral properties of adjacency or Laplacian matrices associated to graphs containing $(m,k,s)$-stars with loops 
and the spectral properties of corresponding graphs containing only $(m,k,s)$-stars without loops.
This approach allows to extend  methods developed for simple graphs also to multigraphs, for example for graph bisection or clustering purposes. 
The results discussed in the paper allow, firstly, to reduce the size of a graph (with or without loops) preserving the spectral properties and then to describe a graph with loops as a simple graph, without discarding relevant information of the original graph. 
As a consequence, it is possible to associate to a graph with loops a Laplacian matrix of the reduced graph without loops.
Despite  the fact that graphs with loops appear in many natural contexts and that they can be obtained by several kinds of aggregation, scaling and blocking procedures, 
they have not been considered as extensively as simple graphs, {since their properties do not verify the conditions required for many theorems on simple graphs}. 
Possible applications of our results could be to organizational networks, where different kinds of ties may appear within the same branch creating loops \cite{baofu2008future}, in citation and co-authorship networks, in which self-citations are possible and the link weights between two authors in co-authorship networks can increase over time if they have further collaborations \cite{Bonzi1991, BARABASI2002590}, and also in opinion networks, where individuals are subject to vanity \cite{RePEc:jas:jasssj:2012-31-2, Quattrociocchi2014OpinionDO}. 
Finally, our results could be relevant in neural network models on undirected graphs, where loops tend to freeze the dynamics that makes the system converge toward fixed points \cite{GOLES2015156}, {and  to general models of anomalous (sub)diffusion on networks, in which loops represent "traps" that slow down the systems dynamics \cite{BouchaudGeorges90}}.

\section*{Acknowledgments}
E. A. thanks Domenico Felice (Max Planck Institute for Mathematics in the Sciences of Leipzig, Germany) for interesting discussions.
Part of this work was developed during E. A.'s stay at Max Planck Institute for Mathematics in the Sciences in Leipzig, the author thanks the institution for the very kind hospitality.
\bibliographystyle{alpha}
\bibliography{loop_20dic18}

\newcommand{\etalchar}[1]{$^{#1}$}
\begin{thebibliography}{PRP{\etalchar{+}}13}

\bibitem[AM85]{Anderson85}
William~N. Anderson and Thomas~D. Morley.
\newblock Eigenvalues of the laplacian of a graph.
\newblock {\em Linear and Multilinear Algebra}, 18(2):141--145, 1985.

\bibitem[ARSB18]{Andreotti18}
Eleonora Andreotti, Daniel Remondini, Graziano Servizi, and Armando Bazzani.
\newblock On the multiplicity of laplacian eigenvalues and fiedler partitions.
\newblock {\em Linear Algebra and its Applications}, 544:206 -- 222, 2018.

\bibitem[Bao08]{baofu2008future}
Peter. Baofu.
\newblock {\em The future of information architecture : conceiving a better way
  to understand taxonomy, network, and intelligence / Peter Baofu}.
\newblock Chandos Oxford, 2008.

\bibitem[BG90]{BouchaudGeorges90}
Jean-Philippe Bouchaud and Antoine Georges.
\newblock Anomalous diffusion in disordered media.
\newblock {\em Physics Reports}, 195(4):131--160, 1990.

\bibitem[BH12]{brouwer12}
Andries~E. Brouwer and Willem~H. Haemers.
\newblock {\em Spectra of Graphs}.
\newblock New York, NY, 2012.

\bibitem[Big93]{biggs:1993}
N.~Biggs.
\newblock {\em Algebraic Graph Theory}.
\newblock Cambridge University Press, 2nd edition, 1993.

\bibitem[BJN{\etalchar{+}}02]{BARABASI2002590}
A.L Barab\'asi, H~Jeong, Z~Néda, E~Ravasz, A~Schubert, and T~Vicsek.
\newblock Evolution of the social network of scientific collaborations.
\newblock {\em Physica A: Statistical Mechanics and its Applications},
  311(3):590 -- 614, 2002.

\bibitem[BLS07]{LaplEigGraphs07}
Turker Biyikoglu, Josef Leydold, and Peter Stadler.
\newblock {\em Laplacien eigenvectors of graphs}.
\newblock Springer, 2007.

\bibitem[Bon76]{Bondy:1976:GTA:1097029}
John~Adrian Bondy.
\newblock {\em Graph Theory With Applications}.
\newblock Elsevier Science Ltd., Oxford, UK, UK, 1976.

\bibitem[BS91]{Bonzi1991}
Susan Bonzi and H.~W. Snyder.
\newblock Motivations for citation: A comparison of self citation and citation
  to others.
\newblock {\em Scientometrics}, 21(2):245--254, Jun 1991.

\bibitem[CH10]{Cohen_Havlin:2010}
Reuven Cohen and Shlomo Havlin.
\newblock {\em {Complex Networks: Structure, Robustness and Function}}.
\newblock Cambridge University Press, August 2010.

\bibitem[Chu97]{Chung97}
F.~R.~K. Chung.
\newblock {\em Spectral Graph Theory}.
\newblock American Mathematical Society, 1997.

\bibitem[DCH13]{RePEc:jas:jasssj:2012-31-2}
Guillaume Deffuant, Timoteo Carletti, and Sylvie Huet.
\newblock The leviathan model: Absolute dominance, generalised distrust, small
  worlds and other patterns emerging from combining vanity with opinion
  propagation.
\newblock {\em Journal of Artificial Societies and Social Simulation}, 16(1),
  2013.

\bibitem[GR01]{Godsil}
C.~{Godsil} and G.~{Royle}.
\newblock {\em Algebraic Graph Theory}, volume 207 of {\em Graduate Texts in
  Mathematics.}
\newblock volume 207 of Graduate Texts in Mathematics. Springer, 2001.

\bibitem[GR15]{GOLES2015156}
Eric Goles and Gonzalo~A. Ruz.
\newblock Dynamics of neural networks over undirected graphs.
\newblock {\em Neural Networks}, 63:156 -- 169, 2015.

\bibitem[Hwa04]{Hwang2004}
Suk-Geun Hwang.
\newblock {Cauchy}'s interlace theorem for eigenvalues of {Hermitian} matrices.
\newblock {\em The American Mathematical Monthly}, 111(2):157--159, 2004.

\bibitem[Mer94]{MERRIS1994143}
Russell Merris.
\newblock Laplacian matrices of graphs: a survey.
\newblock {\em Linear Algebra and its Applications}, 197:143 -- 176, 1994.

\bibitem[MFR16]{ProtMenic16}
Giulia Menichetti, Piero Fariselli, and Daniel Remondini.
\newblock Network measures for protein folding state discrimination.
\newblock {\em Scientific Reports}, 6:30367:1--8, 2016.

\bibitem[New10]{Newman:2010:NI:1809753}
Mark Newman.
\newblock {\em Networks: An Introduction}.
\newblock Oxford University Press, Inc., New York, NY, USA, 2010.

\bibitem[PRP{\etalchar{+}}13]{ProtChemRev13}
Luisa~Di Paola, Micol~De Ruvo, Paola Paci, Daniele Santoni, and Alessandro
  Giuliani.
\newblock Protein contact networks: an emerging paradigm in chemistry.
\newblock {\em Chemical Review}, 113(3):1598--613, 2013.

\bibitem[PW99]{BMSP:BMSP12}
Philippa Pattison and Stanley Wasserman.
\newblock Logit models and logistic regressions for social networks: $ii$.
  multivariate relations.
\newblock {\em British Journal of Mathematical and Statistical Psychology},
  52(2):169--193, 1999.

\bibitem[QCS14]{Quattrociocchi2014OpinionDO}
Walter Quattrociocchi, Guido Caldarelli, and Antonio Scala.
\newblock Opinion dynamics on interacting networks: media competition and
  social influence.
\newblock In {\em Scientific reports}, 2014.

\bibitem[Ray14]{Ray:2014:GTA:2788177}
Santanu~Saha Ray.
\newblock {\em Graph Theory with Algorithms and Its Applications: In Applied
  Science and Technology}.
\newblock Springer Publishing Company, Incorporated, 2014.

\bibitem[Rob13]{ROBINS2013261}
Garry Robins.
\newblock A tutorial on methods for the modeling and analysis of social network
  data.
\newblock {\em Journal of Mathematical Psychology}, 57(6):261 -- 274, 2013.
\newblock Social Networks.

\bibitem[Sco00]{scott00}
J.P. Scott.
\newblock {\em Social Network Analysis: A Handbook}.
\newblock SAGE Publications, January 2000.

\bibitem[Sha15]{Shafie15}
Termeh Shafie.
\newblock A multigraph approach to social network analysis.
\newblock {\em Journal of Social Structure}, 16(1):1 -- 21, 2015.

\bibitem[WF94]{wasserman1994social}
Stanley Wasserman and Katherine Faust.
\newblock {\em Social network analysis: Methods and applications}, volume~8.
\newblock Cambridge university press, 1994.

\end{thebibliography}

\end{document}